\documentclass{amsart}
\pdfoutput=1
\usepackage{microtype}
\usepackage{mathtools}
\usepackage{amsthm}
\usepackage{enumitem}
\usepackage[dvipsnames]{xcolor}
\usepackage{hyperref}
\usepackage{mleftright}

\theoremstyle{plain}
\newtheorem{theorem}{Theorem}[section]
\newtheorem{proposition}[theorem]{Proposition}
\newtheorem{lemma}[theorem]{Lemma}

\theoremstyle{remark}
\newtheorem{remark}[theorem]{Remark}
\theoremstyle{definition}
\newtheorem{definition}[theorem]{Definition}

\begin{document}
\title[Optimal flat ribbons]{Existence of optimal flat ribbons}

\author{Simon Blatt}
\address{Departement of Mathematics\\
University of Salzburg\\
Hellbrunnerstra{\ss}e 34\\
5020 Salzburg\\
Austria}
\email{simon.blatt@sbg.ac.at}

\author{Matteo Raffaelli}
\address{Institute of Discrete Mathematics and Geometry\\
TU Wien\\
Wiedner Hauptstra{\ss}e 8-10/104\\
1040 Vienna\\
Austria}
\email{matteo.raffaelli@tuwien.ac.at}
\thanks{The second-named author was supported by Austrian Science Fund (FWF) project F~77.}
\date{May 2, 2024}
\subjclass[2020]{Primary 49Q10, 53A05; Secondary 49J45, 74B20, 74K20}
\keywords{Bending energy, developable ribbon, direct method in the calculus of variation, Gamma-convergence}

\begin{abstract}
We apply the direct method of the calculus of variations to show that any nonplanar Frenet curve in $\mathbb{R}^{3}$ can be extended to an infinitely narrow flat ribbon having \emph{minimal} bending energy. We also show that, in general, minimizers are not free of planar points, yet such points must be isolated under the mild condition that the torsion does not vanish.
\end{abstract}
\maketitle

\section{Introduction and main result}

In 1930, motivated by the problem of finding the equilibrium shape of a free-standing M\"{o}bius band, Sadowsky~\cite{sadowsky1930, hinz2015} announced without proof that the bending energy $\int_{S} H^{2}\, dA$ of the envelope of the rectifying planes of a $C^{3}$ unit-speed curve $\gamma \colon [0, l] \to \mathbb{R}^{3}$ is given by
\begin{equation}\label{eq:Sadowsky}
\frac{w}{2} \int_{0}^{l} \frac{\mleft(\kappa^{2} + \tau^{2} \mright)^{2}}{\kappa^{2}} \, dt;
\end{equation}
here $w$ is the width of the ribbon, which is measured in the normal plane of $\gamma$ and assumed to be infinitely small, while $\kappa > 0$ and $\tau$ are the curvature and torsion of $\gamma$, respectively. A proof of Sadowsky’s claim was given by Wunderlich in 1962 \cite{wunderlich1962, todres2015}.

Energy~\eqref{eq:Sadowsky} defines a functional on the space of $C^{3}$ curves $\gamma$ in $\mathbb{R}^{3}$, a functional that has attracted a renewed wave of interest in the last twenty years; see, e.g., \cite{hangan2005, chubelaschwili2010, giomi2010, freddi2016, starostin2018, paroni2019, bartels2020, audoly2021, bevilacqua2022, freddi2022}. On the other hand, given a \emph{fixed} curve $\gamma$, it is well known that if the curvature is nowhere zero, then there exist plenty of (infinitely narrow) flat ribbons along $\gamma$. It is therefore natural to interpret Sadowsky's energy formula, or rather a suitable generalization thereof, as a functional on the set of all such ribbons.

An important first step in this direction was taken in \cite{raffaelli2023}, where the author extended Sadowsky's formula to \emph{any} flat ribbon along $\gamma$. Indeed, he showed that the bending energy, in the limit of infinitely small width, is given by 
\begin{equation}\label{eq:Energy}
\frac{w}{2} \int_{J} \frac{\mleft(\kappa_{n}^{2} + \tau_{g}^{2} \mright)^{2}}{\kappa_{n}^{2}} \, dt,
\end{equation}
where $\kappa_{n}$ and $\tau_{g}$ are the normal curvature and the geodesic torsion of $\gamma$, respectively, and where $J = \{ t \in [0, l] \mid \kappa_{n}(t) \neq 0 \}$; see also \cite{dias2014, efrati2015, freddi2016B}.

In this context, a natural question arises: if $\gamma$ is nonplanar, i.e., when the energy is bounded away from zero, does there exist an \emph{optimal} flat ribbon along $\gamma$, that is, one having minimal bending energy? The purpose of this short note is to answer such question affirmatively when $\gamma$ is a Frenet curve. 

A preliminary step in our analysis consists in transforming \eqref{eq:Energy} into a proper functional. To do so, it is enough to observe that (the normals of) any two flat ribbons along the same curve are related by a rotation $\theta \colon [0, l] \to \mathbb{R}$ about the common tangent. 
In particular, when the principal normal $P = \gamma'' /\kappa$ is well-defined, the normal curvature and the geodesic torsion of $\gamma$ with respect to the rotated normal $P(\theta)$ can be expressed by
\begin{align*}
\kappa_{n} &= \kappa \cos(\theta),\\
\tau_{g} &= \theta' + \tau.
\end{align*}
Substituting these relations into \eqref{eq:Energy}, we thus obtain the functional
\begin{equation}\label{eq:EnergyFunctional}
E(\theta)= \int_{0}^{l} \frac{\mleft(\kappa^{2} \cos^{2}\theta + ( \theta' + \tau )^{2} \mright)^{2}}{\kappa^{2}\cos^{2}\theta} \, dt,
\end{equation}
where the integrand is understood to be $0$ (resp, $+\infty$) at any point where both the numerator and denominator vanish (resp., the denominator vanishes but the numerator does not).

Our main result pertains the functional $E$ and is contained in the following theorem.

\begin{theorem}\label{thm:ExistenceOfMinima}
\leavevmode
\begin{enumerate}[font=\upshape]
\item There is a minimizer $\theta_{\min}$ of $E$ on $W^{1,4}([0, l])$, i.e., we have
\begin{equation*}
E(\theta_{\min}) \leq E(\theta) \quad \text{for all $\theta \in W^{1,4}([0, l])$}.
\end{equation*}

\item For any $a, b \in \mathbb{R}$, there is a minimizer of $E$ on the subset
\begin{equation*}
\{ \theta \in W^{1,4}([0, l]) \mid \theta(0) = a \text{ and } \theta(l) = b \}
\end{equation*}
of $W^{1,4}([0, l])$.
\end{enumerate}
\end{theorem}

\begin{remark}
$W^{1,4}((0, l)) \subset C^{0}((0,l))$ by Sobolev embedding theorem, and we interpret $\theta \in W^{1,4}([0,l])$ as the continuous extension of a function in $W^{1,4}((0,l))$.
\end{remark}

\begin{remark}
The theorem remains valid if one replaces $\kappa \in C^{1}([0,l], \mathbb{R}_{>0})$ in \eqref{eq:EnergyFunctional} with any $f \in C^{0}([0,l])$.
\end{remark}
  
The proof of Theorem~\ref{thm:ExistenceOfMinima}, which will be finalized in section~\ref{sec:ExistenceOfMinimizers}, is based on the direct method in the calculus of variation; see section~\ref{sec:GammaConvergence} for an alternative proof relying on $\Gamma$-convergence. It involves showing the coercivity and the weak sequential lower semicontinuity of $E$ on $W^{1,4}([0, l])$ or a suitable closed subset therein. As we explain below, each of these tasks presents some challenge.

First of all, our functional is not coercive on $W^{1,4}([0, l])$, as it is $2\pi$-periodic in $\theta$. Consider, for example, the constant function $\theta_{n} = 2\pi n$: it defines an unbounded sequence in $W^{1,4}([0, l])$, and yet $E(\theta_{n}) = \int_{0}^{l}\frac{(\kappa^{2} + \tau^{2})^{2}}{\kappa^{2}} \, dt$ is (constant and) finite. On the other hand, we can use this periodicity to our advantage: since $W^{1,4}([0,l])$ embeds into the H\"older space $C^{0, \frac 34}([0,l])$---and hence also in $C^0 ([0,l])$---the fact that $E$ is $2\pi$-periodic allows us to only consider functions satisfying $\theta(0) \in [0, 2\pi]$. In the next section we show that $E$ is indeed coercive on the closed subset
\begin{equation*}
V \coloneqq \{\theta \in W^{1,4}([0,l]) \mid \theta(0) \in [0, 2\pi]\}
\end{equation*}
of $W^{1,4}([0,l])$.

As for the sequential lower semicontinuity, the main issue is that our integrand is not continuous. To deal with this problem we use an approximation argument. It turns out, as shown in section~\ref{sec:Semicontinuity}, that the sequential lower semicontinuity of $E$ follows straightforwardly from that of the regular functional 
\begin{equation*}
E_{\varepsilon}(\theta)= \int_{0}^{l} \frac{\mleft(\kappa^{2} \cos^{2}\theta + ( \theta' + \tau )^{2} \mright)^{2}}{\kappa^{2}\cos^{2}\theta + \varepsilon^{2}} \, dt,
\end{equation*}
which for $\varepsilon \to 0$ approximates $E$ monotonically from below.

We emphasize that, precisely because of this discontinuity, the classical indirect method of the calculus of variations does not seem readily applicable in our case. Indeed, to use the Euler--Lagrange equation (in the standard way) one would need to assume that $\theta_{\min}$ is free of \emph{singular} points, i.e., that $\theta_{\min}(t) \notin \pi/2+\pi\mathbb{Z}$ for all $t \in [0,l]$. Our next result confirms that such assumption is, in general, invalid.

\begin{theorem}\label{thm:NumberOfSingularPoints}
Suppose that the torsion $\tau$ is a constant function satisfying 
\begin{equation*}
 \vert \tau \rvert > n \frac \pi l + \max \lvert \kappa \rvert \quad \text{for some $n \in \mathbb{N}_{0}$}.
\end{equation*}
Then the minimizer $\theta_{\min}$ of $E$ in $W^{1,4}([0,l])$ has at least $n$ singular points.
\end{theorem}

Thus, according to Theorem~\ref{thm:NumberOfSingularPoints}, one can enforce the presence of singular points. On the other hand, it turns out that the set of singular points is necessarily discrete when the torsion does not vanish.

\begin{theorem}\label{thm:SingularPointsAreIsolated}
Let $t_0$ be a singular point of $\theta_{\min}$ with $\tau(t_0) \not = 0$. Then $t_0$ is an isolated singular point, i.e., there is an $\varepsilon > 0$ such that the $\varepsilon$-neighborhood $I_{\varepsilon} \coloneqq (t_0 - \varepsilon, t_0 + \varepsilon) \cap [0,l]$ around $t_0$ does not contain any other singular point.
\end{theorem}

It is somewhat surprising that both Theorems \ref{thm:NumberOfSingularPoints} and \ref{thm:SingularPointsAreIsolated} can be obtained, as we do in section~\ref{sec:SingularPoints}, on the basis of such elementary results as the fundamental theorem of calculus, the reverse triangle inequality, and H\"older's inequality.

\section{Coercivity}
Once and for all, let
\begin{equation*}
\Lambda =  \max \{ \lVert \tau \rVert_{L^\infty}, \lVert \kappa \rVert_{L^\infty} \}.
\end{equation*}

The purpose of this section is to prove the following lemma.

\begin{lemma}\label{lem:CoercivityE}
The functional $E$ is coercive on the closed subset
\begin{equation*}
V \coloneqq \{\theta \in W^{1,4}([0,l]) \mid \theta(0) \in [0, 2\pi]\}
\end{equation*}
of $W^{1,4}([0,l])$, i.e., we have for $\theta \in V$ that $E(\theta) \to \infty$ if $\lVert \theta \rVert_{W^{1,4}} \to \infty$. 
\end{lemma}
\begin{proof}
The strategy is to show that $E \to \infty$ if $\lVert \theta' \rVert_{L^{4}} \to \infty$, and that $\lVert \theta' \rVert_{L^{4}} \to \infty$ when $\lVert \theta \rVert_{W^{1,4}} \to \infty$.

First, note that
\begin{equation}\label{eq:Coercivity0}
\Lambda^2 E(\theta) \geq \int_0^l \frac {(\theta'+\tau)^4 } {\cos^{2} \theta} \,dt.
\end{equation}
Since $\lvert \theta' + \tau \rvert \geq \frac{\lvert \theta' \rvert}{2}$ if $\lvert \theta' \rvert \geq 2 \lVert \tau \rVert_{L^{\infty}} \leq 2\Lambda$, equation~\eqref{eq:Coercivity0} implies
\begin{equation*} 
   \Lambda^2 E(\theta) \geq \frac{1}{16} \int_{\lvert \theta' \rvert \geq 2 \Lambda} (\theta')^4 \,dt
    \geq \frac{1}{16} \int_0^l (\theta')^4 \, dt - \Lambda^{4} l = \frac{1}{16} \lVert \theta' \rVert^{4}_{L^{4}} - \Lambda^{4} l,
\end{equation*}
and so if the homogeneous Sobolev norm $\lVert \theta' \rVert_{L^{4}}$ goes to infinity, then so does $E$. 

Next, let $x \neq y \in [0,l]$. Applying H\"older's inequality, we obtain the following Morrey estimate:
\begin{equation*}
\lvert \theta(x) - \theta(y) \rvert =  \lvert \int_{[x,y]} \theta' \, dt \rvert  \leq \lvert x-y \rvert^{\frac 3 4} \lVert \theta' \rVert_{L^4([0,l])}.
\end{equation*}
Together with $\theta(0) \in [0, l]$, this implies
\begin{equation*}
 \lVert \theta\rVert_{L^4} \leq \lVert \theta(\cdot) - \theta (0)\rVert_{L^4} + \lVert \theta(0)\rVert_{L^4} \leq l \lVert \theta'\rVert_{L^4} + 2 \pi l^{\frac 1 4},
\end{equation*}
from which we conclude that 
\begin{equation*} 
\lVert \theta \rVert_{W^{1,4}} \leq (1+l) \lVert \theta' \rVert_{L^4} + 2 \pi l^{\frac 14},
\end{equation*}
as desired.
\end{proof}

\section{Weak sequential lower semicontinuity}\label{sec:Semicontinuity}

In this section we prove the sequential lower semicontinuity of $E$ on $W^{1,4}([0,l])$.

\begin{lemma} \label{lem:SeminContinuity}
The functional $E$ is weakly sequentially lower semicontinuous, i.e., for any sequence of functions $\theta_n \in W^{1,4}([0,l])$ with weak limit $\theta \in W^{1,4}([0,l])$ we have 
\begin{equation*}
E(\theta) \leq \liminf_{n \to \infty} E(\theta_n).
\end{equation*}
\end{lemma}

As already mentioned in the introduction, the plan is to consider for $\varepsilon >0$ the regular functional
\begin{equation*}
E_{\varepsilon}(\theta)= \int_{0}^{l} \frac{\mleft(\kappa^{2} \cos^{2}\theta + ( \theta' + \tau )^{2} \mright)^{2}}{\kappa^{2}\cos^{2}\theta + \varepsilon^{2}} \, dt.
\end{equation*}
Note that, for fixed $\theta$, the integrand is monotonically decreasing in $\varepsilon$. Using Beppo Levi's monotone convergence theorem, we thus get that
\begin{equation*}
E(\theta) = \lim_{\varepsilon \downarrow 0} E_{\varepsilon} (\theta) = \sup_{\varepsilon >0} E_{\varepsilon} (\theta).
\end{equation*}

Now, if we knew that $E_{\varepsilon}$ was weakly sequentially lower semicontinuous, then the following well-known lemma would imply the same for $E$. It simply states the the supremum of any collection of lower semicontinuous functions is again lower semicontinuous.

\begin{lemma}
Let $X$ be a topological space, and let $A_i \colon X \to [0, \infty]$, $i \in I$ be a family of lower semicontinuous functions. Then $A \colon X \rightarrow [0, \infty]$ defined by $A(x) = \sup_{i \in I} A_i (x)$ is lower semicontinuous.
\end{lemma}


\begin{proof}
We need to show that the set $A^{-1} ((a,\infty])$ is open for all $a \in [0, \infty)$. First, using the fact that $A= \sup_{i \in I} A_i$, we get
\begin{equation*}
A^{-1} ((a,\infty]) = \bigcup_{i \in I} A^{-1}_i ((a, \infty]).
\end{equation*}
Since $A_i$ is lower semicontinuous for all $i \in I$, we observe  that $A^{-1} ((a,\infty])$ is the union of open sets, and so open itself.
\end{proof}

It remains to show the lower semicontinuity of $E_\varepsilon$.

\begin{lemma}\label{lem:LowerSemicontinuity}
For all $\varepsilon >0$, the functional $E_\varepsilon$ is weakly sequentially lower semicontinuous on $W^{1,4}([0,l])$.
\end{lemma}

As the integrand is convex in $\theta'$, to prove Lemma~\ref{lem:LowerSemicontinuity} it would be enough to invoke \cite[Theorem~1.6]{struwe2008}. We nevertheless include an independent proof---which combines Sobolev embeddings with Mazur's lemma---for the benefit of the reader.

\begin{proof}[Proof of Lemma~\textup{\ref{lem:LowerSemicontinuity}}]
Let $\theta_{n}$ converge weakly to $\theta$ in $W^{1,4}([0,l])$. Since $W^{1,4}([0,l])$ embeds compactly into $C^{0}([0,l])$ and the image of any weakly convergent sequence under a compact embedding converges strongly, we deduce that $\theta_{n}$ and $(\kappa^{2} \cos^{2}\theta_{n}+\varepsilon^{2})^{-1}$ converge uniformly to $\theta$ and $(\kappa^{2} \cos^{2}\theta+\varepsilon^{2})^{-1}$, respectively. Moreover, exchanging $\theta_{n}$ by a suitable subsequence, we may assume that $E_{\varepsilon}(\theta_{n})$ converges to the limit inferior of the original sequence.

We now aim to get rid of $\theta_{n}$ in the expression of $E_{\varepsilon}(\theta_{n})$ and only keep $\theta'_{n}$. To this end, we first rewrite $E_{\varepsilon}(\theta_{n})$
as
\begin{equation*}
E_{\varepsilon}(\theta_{n}) = I_{1}(\theta_{n}) + I_{2}(\theta_{n}) + I_{3}(\theta_{n}) + I_{4}(\theta_{n}),
\end{equation*} 
where
\begin{align*}
I_{1}(\theta_{n}) &= \int_0^l \mleft(\kappa^2 \cos^{2}\theta_n + (\theta'_n + \tau)^2 \mright)^2  \mleft(\frac{1}{ \kappa^2 \cos^{2}\theta_n + \varepsilon^2 } - \frac {1}{ \kappa^2 \cos^{2}\theta + \varepsilon^{2} } \mright) \,dt,\\
I_{2}(\theta_{n}) &=  \int_0^l \frac{\kappa^4 \cos^{4}\theta_n - \kappa^{4} \cos^{4}\theta }{ \kappa^2 \cos^{2}\theta + \varepsilon^2 } \,dt,\\
I_{3}(\theta_{n}) &= \int_0^l \frac {2 \mleft(\kappa^2\cos^{2}\theta_n  - \kappa^2 \cos^{2}\theta  \mright) (\theta_n'+\tau)^2 }{ \kappa^2 \cos^{2}\theta + \varepsilon^2 } \, dt,\\
I_{4}(\theta_{n}) &= \int_0^l \frac {\mleft(\kappa^2\cos^{2}\theta + (\theta_n'+\tau)^2\mright)^2 }{ \kappa^2\cos^{2}\theta + \varepsilon^2 } \,dt.   
\end{align*}
Since $\cos \theta_n$ and $(\kappa^2 \cos^{2}\theta_n + \varepsilon^2)^{-1}$ converge uniformly to  $\cos \theta$ and $(\kappa^2\cos^{2}\theta + \varepsilon^2)^{-1}$, respectively, and
\begin{equation*}
 \int_0^l \mleft(\kappa^2\cos^{2}\theta_n + (\theta_n'+\tau)^2\mright)^2 \,dt \leq (\Lambda^{-2} + \varepsilon^2)^{-1} \sup_n E_\varepsilon( \theta_n) < \infty,
\end{equation*} 
H\"older's inequality implies that $I_1(\theta_{n})$, $I_2(\theta_{n})$, and $I_3(\theta_{n})$ converge to $0$ as $n$ tends to $\infty$; hence that $I_4(\theta_{n})$ converges to $\lim_{n \to \infty} E_{\varepsilon}(\theta_n)$.

To deal with the term $I_4$ we use that the integrand is convex in $\theta'_n$. Let $m \in \mathbb{N}$. Mazur's lemma~\cite[Theorem~3.13]{rudin1991} tells us that there are convex combinations $\phi_n = \sum_{i=m}^{m+n} \lambda_{n,i} \theta_i$ of $\theta_m, \dotsc, \theta_{m+n}$ that converge strongly to $\theta$ in $W^{1,4}([0,l])$. Using the convexity of the integrand, we obtain
\begin{equation*}
\sup_{i \geq m } I_4(\theta_i) \geq \sum_{i=m}^{m+n} \lambda_{n,i} I_4(\theta_i) \geq I_4 (\sum_{i=m}^{m+n} \lambda_{n,i} \theta_i ) = I_4(\phi_n).
\end{equation*} 
As $I_{4}$ is continuous on $W^{1,4}([0,l])$ and the convex combinations $\phi_{n}$ converge to $\theta$ strongly, this yields
\begin{equation*}
\sup_{i \geq m} I_4(\theta_i) \geq I_4 (\theta).
\end{equation*} 
Finally, taking the limit as $m \to \infty$, we deduce that 
\begin{equation*}
\lim_{m \to\infty} E_{\varepsilon} (\theta_m) = \lim_{m \to \infty} I_4 (\theta_m) \geq I_4 (\theta).
\end{equation*}
Hence $E_4$ is weakly sequentially lower semicontinuous on $W^{1,4}([0,l])$.
\end{proof}

\section{Existence of minimizers}\label{sec:ExistenceOfMinimizers}

We are now ready to prove our main result, Theorem~\ref{thm:ExistenceOfMinima} in the introduction.

\begin{proof}[Proof of Theorem~\textup{\ref{thm:ExistenceOfMinima}}]
To begin with, let 
\begin{equation*}
V' = \{ \theta \in W^{1,4}([0, l]) \mid \theta(0) = a - 2\pi \lfloor a \rfloor \text{ and } \theta(l) = b -2\pi \lfloor a \rfloor\},
\end{equation*}
where $\lfloor \cdot \rfloor$ denotes the floor function. We are going to show that $E$ has a minimizer on both $V$ and its subset $V'$, which is also closed and convex in $W^{1,4}([0,l])$. This way, the statement will follow directly from the periodicity of $E$.

Let thus $\theta_{n}$ be a minimizing sequence for $E$ on either $V$ or $V'$. By Lemma~\ref{lem:CoercivityE}, the sequence $\theta_{n}$ is bounded in $W^{1,4}([0,l])$, and so, according to \cite[Lemma~1.13.3]{megginson1998}, we can assume after passing to a subsequence that it converges weakly to a function $\theta_{0} \in W^{1,4}([0,l])$; in particular, since any closed and convex subset of a Banach space is weakly closed, we deduce that the limit $\theta_{0}$ is contained in either $V$ or $V'$.

Now, as $E$ is weakly sequentially lower semicontinuos by Lemma~\ref{lem:LowerSemicontinuity}, we have 
\begin{equation*}
\inf_{\theta} E(\theta) = \lim_{n \to \infty} E(\theta_n) \geq E(\theta_{0}) \geq \inf_{\theta} E(\theta),
\end{equation*}
where the infimum is taken over $V$ or $V'$. Hence these inequalities must be equalities, and so $\theta_{0}$ is a minimizer of $E$ on $V$ or $V'$. 
\end{proof}

\section{\texorpdfstring{$\Gamma$}{Gamma}-convergence}\label{sec:GammaConvergence}

Here we give an alternative proof of Theorem~\ref{thm:ExistenceOfMinima} based on the fundamental theorem of $\Gamma$-convergence~\cite[Theorem~7.8]{dalmaso1993}. 

To begin with, note that $E_{\varepsilon}$ is coercive on $V$, as $E_{\varepsilon} > E$; having already shown that it is weakly sequentially lower semicontinuous, we have the existence of minimizers for any $\varepsilon > 0$.

\begin{proposition}
\leavevmode
\begin{enumerate}[font=\upshape]
\item There is a minimizer of $E_{\varepsilon}$ on $W^{1,4}([0, l])$.

\item For any $a, b \in \mathbb{R}$, there is a minimizer of $E_{\varepsilon}$ on the subset
\begin{equation*}
W^{1,4}_{ab}([0,l]) = \{ \theta \in W^{1,4}([0, l]) \mid \theta(0) = a \text{ and } \theta(l) = b \}
\end{equation*}
of $W^{1,4}([0, l])$.
\end{enumerate}
\end{proposition}

To apply the fundamental theorem, we first show that $E_{\varepsilon} \xrightarrow{\Gamma} E$ weakly.

\begin{proposition}
The functionals $E_{\varepsilon}$ $\Gamma$-converge to $E$ on $W^{1,4}([0,l])$ with the weak topology as $\varepsilon$ goes to $0$.
\end{proposition}
\begin{proof}
To show the $\Gamma$-convergence of  $E_{\varepsilon}$, we have to prove a liminf and a limsup inequality.

For the liminf inequality, suppose that $\theta_{\varepsilon}$ converges weakly to $\theta$ in $W^{1,4}([0,l])$, and choose a null sequence $\varepsilon_{n}$ such that
\begin{equation*}
\lim_{n \rightarrow \infty} E_{\varepsilon_n} (\theta_{\varepsilon_n}) = \liminf_{\varepsilon \downarrow 0 } E_\varepsilon (\theta_{\varepsilon}).
\end{equation*}
Clearly, as the functionals $E_{\varepsilon_{n}}$ are weakly sequentially lower semicontinuous, 
\begin{equation*}
\lim_{n\rightarrow \infty} E_{\varepsilon_m} (\theta_{\varepsilon_n}) \geq E_{\varepsilon_m} (\theta) \quad \text{for all $m \in \mathbb{N}$}.
\end{equation*}
Letting $m$ go to infinity and observing that the right-hand side converges to $E(\theta)$, we thus obtain
\begin{equation*}
\lim_{n\rightarrow \infty} E_{\varepsilon_n} (\theta_{\varepsilon_n}) \geq E (\theta),
\end{equation*}
as desired.

As for the limsup inequality, for $\theta \in W^{1,4}([0,l])$ we simply take $\theta_\varepsilon = \theta$ for all $\varepsilon >0$ as recovery sequence. This we can do, because the monotonicity of the integrand implies
\begin{equation*}
\lim_{\varepsilon \downarrow 0}  E_\varepsilon (\theta) = E(\theta)
\end{equation*}
via Beppo Levi's monotone convergence theorem.
\end{proof}

Having shown that the approximating functionals $E_{\varepsilon}$ have a minimizer and  $\Gamma$-converge to $E$, the missing ingredient needed to deduce the existence of a minimizer of $E$ is equicoercivity, which we discuss below.

\begin{definition}
A family of functionals $F_\alpha \colon X \to \mathbb R$, $\alpha \in I$ on a normed vector space $X$ is said to be \textit{equicoercive} if the set 
\begin{equation*} 
\bigcup_{\alpha \in I} \{F_\alpha \leq t\} 
\end{equation*}
is bounded for all $t \in \mathbb{R}$.
\end{definition}

\begin{lemma}\label{lem:Equicoercivity}
The family of functionals $E_\varepsilon$, $0 < \varepsilon \leq 1$ is equicoercive on $V$.
\end{lemma}
\begin{proof}
As $E_\varepsilon$ is pointwise nonincreasing in $\varepsilon >0$, we have 
\begin{equation*} 
E_\varepsilon \geq E_1,
\end{equation*}
and so
\begin{equation*} 
\bigcup_{\varepsilon \in (0,1]} \{E_\varepsilon \leq t\} \subset \{E_1 \leq t\}.
\end{equation*}
But the set on the right-hand side is bounded, as we know that $E_1$ is coercive on $W^{1,4}([0,l])$.
\end{proof}

Applying \cite[Theorem~7.8]{dalmaso1993}, we finally get the existence of a minimizer of $E$.

\begin{theorem}\label{thm:GammaConvergence}
The infimum of $E$ is attained and we have
\begin{equation*} 
\min_{\theta} E(\theta) = \lim_{\varepsilon \downarrow 0} \min_{\theta} E_{\varepsilon}(\theta),
\end{equation*}
where the minima are taken over $W^{1,4}([0,l])$ or $W^{1,4}_{ab}([0,l])$.
\end{theorem}

To keep the exposition as self-contained as possible, we close this section by giving an independent proof of Theorem~\ref{thm:GammaConvergence}.

\begin{proof}[Proof of Theorem~\textup{\ref{thm:GammaConvergence}}]
From $E_\varepsilon \leq E$ we immediately get the inequality
\begin{equation}\label{eq:LimMinInequality}
\inf_{\theta} E(\theta) \geq \lim_{\varepsilon \downarrow 0} \min_{\theta} E_{\varepsilon}(\theta).
\end{equation}
Moreover, as $E(0) = \int_0^l \kappa ^2 + \tau^2 \,dt \leq \Lambda ^2 l + \lVert\tau\rVert^2_{L^2}$ and $E_{\varepsilon}$ is monotonically decreasing in $\varepsilon > 0$, we have the uniform bound
\begin{equation*}
\min_{\theta} E_{\varepsilon}(\theta) \leq \Lambda^2 l + \lVert \tau\rVert^2_{L^2}.
\end{equation*}

To show the existence of a minimizer of $E$, let $\varepsilon_n > 0$ be a null sequence, and choose minimizers $\theta_{\varepsilon_{n}} \in V$ (resp., $\theta_{\varepsilon_{n}} \in V'$) of $E_{\varepsilon_n}$. As $E_{\varepsilon_n} (\theta_{\varepsilon_{n}}) \leq \Lambda ^2 l+ \lVert\tau\rVert_{L^2}^2$, Lemma~\ref{lem:Equicoercivity} ensures that the sequence $\theta_{\varepsilon_{n}}$ is bounded in $W^{1,4}([0,l])$. Hence, exactly as in the proof of Theorem~\ref{thm:ExistenceOfMinima}, we can assume that it converges weakly to some $\theta_0 \in V$ (resp., $\theta_0 \in V'$) in $W^{1,4}([0,l])$. 

Now, applying the liminf inequality, we obtain 
\begin{equation*}
\inf_{\theta}E(\theta) \leq E(\theta_0)  \leq \liminf_{n\rightarrow \infty} E_{\varepsilon_n} (\theta_{\varepsilon_{n}}) = \lim_{\varepsilon \downarrow 0} \min_{\theta} E_{\varepsilon}(\theta),
\end{equation*}
where the infimum and minimum are taken over $V$ (resp., $V'$). Together with \eqref{eq:LimMinInequality}, this implies 
\begin{align*}
	\lim_{\varepsilon \downarrow 0} \min_{\theta} E_{\varepsilon}(\theta) & \leq \inf_{\theta}E(\theta) \leq E(\theta_0)  \leq \liminf_{n\rightarrow \infty} E_{\varepsilon_n} (\theta_{\varepsilon_{n}}) \\ &= \lim_{\varepsilon \downarrow 0} \min_{\theta} E_{\varepsilon}(\theta),
\end{align*}
and so this series of inequalities must hold with equality. Especially, we have 
\begin{equation*}
E(\theta_0)  = \min_{\theta} E (\theta) = \lim_{\varepsilon \downarrow 0} \min_{\theta} E_{\varepsilon}(\theta),
\end{equation*}
and the theorem follows by periodicity.
\end{proof}

\section{Number of singular points}\label{sec:SingularPoints}

Let $t \in [0, l]$. We say that $t$ is a \textit{singular} point of $\theta \in W^{1,4}([0,l])$ if $\theta(t) \in \frac{\pi}{2} + \pi\mathbb{Z}$, i.e., if the denominator of our integrand vanishes. Clearly, when $E(\theta) <\infty$, singular points of $\theta$ correspond to planar points of the associated flat ribbon. The purpose of this section is to show that under certain assumptions, a minimizer has ``many" singular points; besides, we will see that if $\tau(t) \neq 0$, then $t$ is at most an \emph{isolated} singular point.

We begin with a lemma. To state it, let us introduce for any compact interval $I \subset [0, l]$ the energy
\begin{equation*}
E_{I}(\theta)= \int_{I} \frac{\mleft(\kappa^{2} \cos^{2}\theta + ( \theta' + \tau )^{2} \mright)^{2}}{\kappa^{2}\cos^{2} \theta} \, dt.
\end{equation*}

\begin{lemma}\label{lem:Singularities}
Let $a, b$ such that $0 \leq a <b \leq l$, and suppose that the torsion $\tau$ never vanishes in $[a,b]$. Then
\begin{equation}\label{eq:SingularitiesLemma}
\lvert\theta(a) - \theta(b)\rvert \geq A (b-a)  -  B ^{ \frac 12}  (b-a)^{\frac 34 }E_I (\theta) ^{\frac 1 4},
\end{equation}
where
\begin{equation*}
A = \min_{t\in [a,b]} \lvert \tau (t)\rvert
\end{equation*}
and
\begin{equation*}
B= \max_{t \in [a,b]}\lvert\kappa(t) \cos \theta(t) \rvert.
\end{equation*}
\end{lemma}

\begin{proof}
First, an application of the fundamental theorem of calculus and the reverse triangle inequality yields 
\begin{equation*}
\lvert \theta(a) - \theta(b) \rvert = \lvert \int_a^b \theta'\,dt \rvert =  \lvert \int_a^b - \tau + (\theta' + \tau) \, dt \rvert \geq \lvert \int_a^b \tau\, dt \rvert - \int_a^b \lvert\theta' + \tau \rvert \, dt.
\end{equation*}
Note that, as $\tau$ never vanishes in $[0,l]$, it must have a sign there. Hence
\begin{equation*}
\lvert \int_a^b \tau \,dt\rvert \geq (b-a) \min_{t\in [a,b]} \lvert \tau (t) \rvert = (b-a) A .
\end{equation*}
Moreover, by H\"older's inequality,
\begin{align*}
\int_a^b \lvert \theta'+ \tau \rvert \, dt 
&\leq (b-a)^{\frac 34 }\biggl (\int_a^b \lvert \theta' + \tau \rvert^4 \, dt \biggr)^{\frac 1 4}
\\ & \leq \max_{t \in [a,b]} \lvert \kappa(t) \cos \theta(t) \rvert^{ \frac 12}  (b-a)^{\frac 34 } \biggl(\int_a^b \frac { \lvert \theta' + \tau \rvert^4}{\kappa^2 \cos^{2} \theta} \,dt \biggr)^{\frac 1 4}
\\ & \leq  \max_{t \in [a,b]}\lvert \kappa(t) \cos \theta(t) \rvert ^{\frac 12} (b-a)^{\frac 34 } E_I (\theta) ^{\frac 1 4} 
\\ &= B ^{ \frac 12} (b-a)^{\frac 34 } E_I (\theta) ^{\frac 1 4}.	
\end{align*}
Summing up, we get
\begin{equation*}
\lvert \theta(a) - \theta (b)\rvert  \geq  A (b-a)  - B ^{ \frac 12} (b-a)^{\frac 34 } E_I (\theta) ^{\frac 1 4},
\end{equation*}
which is the desired conclusion.
\end{proof}

Applying Lemma~\ref{lem:Singularities} for $a=0$ and $b=l$, we can now deduce that minimizers of $E$ are generally not free of singular points. In fact, by making sure that the right-hand side of \eqref{eq:SingularitiesLemma} is large enough, one can enforce the presence of \emph{any} given number of singular points---as explained by Theorem~\ref{thm:NumberOfSingularPoints} in the introduction (reproduced below for the reader's convenience).

\begingroup
\def\thetheorem{\ref*{thm:NumberOfSingularPoints}}
\begin{theorem}
Suppose that the torsion $\tau$ is a constant function satisfying 
\begin{equation*}
	\vert \tau \rvert > n \frac \pi l + \max \lvert \kappa \rvert \quad \text{for some $n\in \mathbb{N}_{0}$}.
\end{equation*}
Then the minimizer $\theta_{\min}$ of $E$ in $W^{1,4}([0,l])$ has at least $n$ singular points.
\end{theorem}
\addtocounter{theorem}{-1}
\endgroup

\begin{proof}
Let $K = \max_{t\in [0,l]} \lvert \kappa(t) \rvert$. Using the linear function $\theta_0(t) = - \int_0^t \tau (s) \,ds$ as competitor, we obtain 
\begin{equation*}
E(\theta_{\min}) \leq E(\theta_0) = \int_0^l \kappa ^2 \cos^{2} \theta\, dt \leq K^2 l.
\end{equation*}

Now we apply Lemma~\ref{lem:Singularities} to the complete interval, i.e., for $a=0$ and $b=l$. Since $A= \min_{t\in [0,l]} \lvert\tau (t)\rvert = \lvert\tau \rvert$ and $B= \max_{t\in [0,l]} \lvert\kappa(t) \cos (\theta(t)) \rvert \leq K$, equation~\eqref{eq:SingularitiesLemma} yields
\begin{equation*}
\lvert \theta_{\min}(0) - \theta_{\min}(l) \rvert \geq l  \lvert \tau \rvert - K^{\frac 1 2} l^{\frac 34} K^{\frac 1 2} l^{\frac 14} = l ( \lvert \tau \rvert - K),
\end{equation*}
which is larger than $n \pi$ if $\lvert \tau \rvert > n \frac \pi l + K$.
\end{proof}

Although, as we just saw, singularities abound, the following generalized version of Theorem~\ref{thm:SingularPointsAreIsolated} shows that they typically form a discrete set. The proof is again based on Lemma~\ref{lem:Singularities}.

\begin{theorem}\label{thm:SingularPointsAreIsolated2}
Suppose that $E(\theta) < \infty$, and let $t_0$ be a singular point of $\theta$ with $\tau(t_0) \not = 0$. Then $t_0$ is an isolated singular point, i.e., there is an $\varepsilon > 0$ such that the $\varepsilon$-neighborhood $I_{\varepsilon} \coloneqq (t_0 - \varepsilon, t_0 + \varepsilon) \cap [0,l]$ around $t_0$ does not contain any other singular point.
\end{theorem}

The simple heuristic behind the proof is that, in the neighborhood of a singular point, we must have 
\begin{equation*}
\theta' + \tau \approx 0,
\end{equation*}
as otherwise the energy cannot be bounded. So when $\tau$ does not vanish, the function $\theta$ must be strictly monotone.

\begin{proof}[Proof of Theorem~\textup{\ref{thm:SingularPointsAreIsolated2}}]
Let $t_0$ be a singular point of $\theta$, and choose $\varepsilon >0$ such that $\lvert \theta(t) - \theta(t_0) \rvert \leq \pi$ for all $t\in I_\varepsilon$. Noting that, by H\"older's inequality,   
\begin{equation*}
\lvert \theta(x) - \theta(y)\rvert \leq \lvert x-y \rvert^{\frac 34} \lVert \theta'\rVert_{L^4([x,y])} \leq \lvert x-y \rvert^{\frac 34} \Lambda^{\frac 12} E^{\frac 1 4}_{[x,y]} (\theta),
\end{equation*}
we first obtain, using that $\cos$ is Lipschitz continuous with Lipschitz constant one,
\begin{align*}
	\lvert\cos(\theta(t))\rvert & = \lvert\cos(\theta(t)) - \cos(\theta(t_0))\rvert \leq  \lvert\theta(t) - \theta(t_0)\rvert \leq \lvert t-t_0 \rvert^ {\frac 34} \Lambda^{\frac 1 2} E^{\frac 1 4}_{[t_0,t]}(\theta) \\ & \leq \lvert t- t_0\rvert ^{\frac 34} \Lambda^{ \frac 1 2} E^{\frac 1 4}_{[t-\varepsilon, t+\varepsilon]}(\theta).
\end{align*}

Then, as 
\begin{equation*}
B= \max_{x \in [t_0,t]} \lvert\kappa(x) \cos( \theta(x)) \rvert \leq \lvert t-t_0 \rvert^{\frac 3 4} \Lambda ^{\frac 1 2} E^{\frac 1 4}(\theta),
\end{equation*}
an application of Lemma~\ref{lem:Singularities} with $[a,b] = [t_0,t]$ gives
\begin{align*}
	\lvert \theta(t) - \theta(t_0)\rvert &\geq \lvert t-t_0 \rvert \min_{x \in [t_0- \varepsilon, t_0 + \varepsilon] \cap [0,l]} \lvert \tau(x)\rvert - \Lambda^{\frac 1 4} \lvert t-t_0 \rvert^{\frac 3 4 + \frac 3 8} E^{\frac 14 + \frac 1 8} (\theta ) 
	\\ & = \lvert t-t_0 \rvert \bigl(\min_{x \in [t_0- \varepsilon, t_0 + \varepsilon] \cap [0,l]} \lvert \tau(x)\rvert - \Lambda^{\frac 1 4} \lvert t-t_0 \rvert^{\frac 1 8} E^{\frac 3 8} (\theta )\bigr)
	\\ & \geq \lvert t-t_0 \rvert \bigl(\min_{x \in [t_0- \varepsilon, t_0 + \varepsilon] \cap [0,l]} \lvert \tau(x) \rvert - \Lambda^{\frac 1 4} \varepsilon^{\frac 1 8} E^{\frac 3 8} (\theta ) \bigr).
\end{align*}

Finally, since
\begin{equation*}
\min_{x \in [t_0- \varepsilon, t_0 + \varepsilon] \cap [0,l]} \lvert \tau(x) \rvert - \Lambda^{\frac 1 4} \varepsilon^{\frac 1 8} E^{\frac 3 8} (\theta ) \to \lvert\tau(t_0)\rvert >0 \quad \text{as $\varepsilon \to 0$},
\end{equation*}
it follows that there is an $\varepsilon >0$ such that 
\begin{equation*}
\lvert \theta(t) - \theta(t_0) \rvert  >0 \quad \text{for all $t \in I_\varepsilon$ with $t\neq t_0$}.
\end{equation*}
This shows that $I_\varepsilon$ does not contain any other singular point, as desired.	
\end{proof}

\section*{Acknowledgment}
We thank an anonymous referee for several corrections.

\bibliographystyle{amsplain}
\bibliography{biblio}
\end{document}